\documentclass[10pt]{amsart}
  
\usepackage[T1]{fontenc}
\usepackage[small,it]{caption}
\usepackage[english]{babel}
\usepackage{url}
\usepackage{fullpage}
\usepackage{graphicx}                         
\usepackage{amsmath}
\usepackage{amsthm}                      
\usepackage{amssymb} 
\usepackage{mathrsfs}
\usepackage{stmaryrd}
\usepackage{enumerate}
\usepackage{tikz}

\usepackage{tikz-cd}

\usepackage[margin=3cm]{geometry}
\usepackage{euscript}
\renewcommand{\mathcal}{\EuScript}
\swapnumbers
\theoremstyle{plain}
\makeatletter
\def\swappedhead#1#2#3{%
	\thmnumber{\@upn{\the\thm@headfont#2\@ifnotempty{#1}{.~}}}%
	\thmname{#1}%
	\thmnote{ {\the\thm@notefont(#3)}}}
\makeatother

\usepackage{color}
\usepackage{hyperref}
\usepackage[noabbrev,capitalize]{cleveref}
\hypersetup{
	colorlinks,
	linkcolor={red!50!black},
	citecolor={blue!50!black},
	urlcolor={blue!80!black}
}

\DeclareMathSymbol{A}{\mathalpha}{operators}{`A}
\DeclareMathSymbol{B}{\mathalpha}{operators}{`B}
\DeclareMathSymbol{C}{\mathalpha}{operators}{`C}
\DeclareMathSymbol{D}{\mathalpha}{operators}{`D}
\DeclareMathSymbol{E}{\mathalpha}{operators}{`E}
\DeclareMathSymbol{F}{\mathalpha}{operators}{`F}
\DeclareMathSymbol{G}{\mathalpha}{operators}{`G}
\DeclareMathSymbol{H}{\mathalpha}{operators}{`H}
\DeclareMathSymbol{I}{\mathalpha}{operators}{`I}
\DeclareMathSymbol{J}{\mathalpha}{operators}{`J}
\DeclareMathSymbol{K}{\mathalpha}{operators}{`K}
\DeclareMathSymbol{L}{\mathalpha}{operators}{`L}
\DeclareMathSymbol{M}{\mathalpha}{operators}{`M}
\DeclareMathSymbol{N}{\mathalpha}{operators}{`N}
\DeclareMathSymbol{O}{\mathalpha}{operators}{`O}
\DeclareMathSymbol{P}{\mathalpha}{operators}{`P}
\DeclareMathSymbol{Q}{\mathalpha}{operators}{`Q}
\DeclareMathSymbol{R}{\mathalpha}{operators}{`R}
\DeclareMathSymbol{S}{\mathalpha}{operators}{`S}
\DeclareMathSymbol{T}{\mathalpha}{operators}{`T}
\DeclareMathSymbol{U}{\mathalpha}{operators}{`U}
\DeclareMathSymbol{V}{\mathalpha}{operators}{`V}
\DeclareMathSymbol{W}{\mathalpha}{operators}{`W}
\DeclareMathSymbol{X}{\mathalpha}{operators}{`X}
\DeclareMathSymbol{Y}{\mathalpha}{operators}{`Y}
\DeclareMathSymbol{Z}{\mathalpha}{operators}{`Z}

\usepackage{palatino}
\usepackage{mathtools}
\newtheorem{thm}{Theorem}[section]

\newtheorem{prop}[thm]{Proposition}

\theoremstyle{definition}

\newtheorem{rem}[thm]{Remark}

\newtheorem{defn}[thm]{Definition}

\newtheorem{para}[thm]{}

\usepackage[backend=bibtex,maxnames=7,maxalphanames=7,style=ext-alphabetic,articlein=false,doi=false,isbn=false,url=false,eprint=false]{biblatex}
\bibliography{../database.bib}
\AtBeginBibliography{\small}

\newcommand{\open}{\mathrm{Open}}
\newcommand{\op}{\mathrm{op}}
\newcommand{\Z}{\mathbf Z}
\newcommand{\sing}{\mathrm{sing}}
\newcommand{\sheaf}{\mathrm{sheaf}}
\newcommand{\Sing}{\mathcal S}

\DeclareMathOperator*{\hocolim}{hocolim}

\DeclareMathOperator*{\holim}{holim}
\title{A remark on singular cohomology and sheaf cohomology}

\author{Dan Petersen}
\email{dan.petersen@math.su.se}
\thanks{The author acknowledges support by ERC-2017-STG 759082 and a Wallenberg Academy Fellowship. }
\begin{document}

\begin{abstract}We prove a comparison isomorphism between singular cohomology and sheaf cohomology. 
\end{abstract}
	\maketitle

\section{Introduction}

\begin{para}Let $X$ be a topological space, $A$ an abelian group. We say that $X$ is \emph{cohomologically locally connected}\footnote{The notion of being cohomologically locally connected is classically defined by the stronger condition that for all $k \in \Z$, all $x \in X$ and all neighborhoods $x \in U$, there exists a smaller neighborhood $x \in V \subset U$ such that $H^k_\sing(U,x;A) \to H^k_\sing(V,x;A)$ is the zero map. The definition used here seems more natural. 
} (with respect to $A$) if for all $x \in X$ and all $k \in \mathbf Z$, 
$$ \varinjlim_{U \ni x}H^k_\sing(U,x;A)=0.$$ 
The goal of this note is to give a homotopy-theoretic proof of the following result:
\end{para}
                                      
\begin{thm}Suppose that the topological space $X$ is cohomologically locally connected with respect to $A$. Then sheaf cohomology and singular cohomology of $X$ with coefficients in $A$ are canonically isomorphic. 
\end{thm}

\begin{para}\label{intro paragraph}It is of course well known that singular cohomology coincides with sheaf cohomology for `reasonable' spaces. I will offer two excuses for having written this paper:
\begin{enumerate}[(i)]
\item The proof given here seems to me to be the `right' one. Although the argument is not simple in the sense of being elementary or self-contained, it is simple in that it only uses general principles and makes the result seem inevitable; there is no cleverness involved.
\item In the classical literature one finds a proof of this result under the assumption that $X$ is cohomologically locally connected and \emph{paracompact Hausdorff}. It is a surprisingly recent theorem of Sella\footnote{Sella makes the stronger assumption that $X$ is \emph{semi-locally contractible}, but his proof only requires $X$ to be cohomologically locally connected.} \cite{sella} that the result is true also without the assumption that $X$ is paracompact Hausdorff. However, Sella's argument is quite intricate, and reading it makes one wonder what is `really' going on. 
\end{enumerate} 
We also give a quick argument for the more refined statement that if $A$ is a commutative ring then there is a canonical isomorphism of $E_\infty$-algebras $\mathrm R\Gamma(X,A) \simeq C^\bullet_\sing(X,A)$ in the $\infty$-category of cochain complexes.\end{para}


\begin{para}\label{outline}Let us first of all outline the standard proof of the {comparison} isomorphism between sheaf and singular cohomology, which can be found with some variations in many textbooks. Denote by $\Sing$ the cochain complex of presheaves on $X$ which in degree $k$ is given by $U \mapsto C_\sing^k(U,A)$, and for a presheaf $F$ denote by $\underline F$ its sheafification. The proof proceeds as follows:
\begin{enumerate}[(i)]
\item Let $A$ be the constant presheaf. Argue that the unit map $A \to \Sing$ is a quasi-isomorphism on stalks. 
\item Apply the sheafification functor to get a quasi-isomorphism $\underline A \to \underline \Sing$. 
\item Argue that the sheaves $\underline \Sing^k$ are acyclic, so that $\underline A \to \underline \Sing$ is an acyclic resolution of the constant sheaf.
\item Argue that the map $\Gamma(X,\Sing) \to \Gamma(X,\underline \Sing)$ is a quasi-isomorphism.
\end{enumerate}
\end{para}

\begin{para}\label{proofsketch}
Each of the steps (i), (iii) and (iv) of \S\ref{outline} require some point-set hypotheses on $X$. Step (i) is of course equivalent to $X$ being cohomologically locally connected. Steps (iii) and (iv) require $X$ to be paracompact Hausdorff. For step (iii) a natural idea is to observe that the presheaves $\Sing^k$ are \emph{flasque}. Unfortunately, the sheafification of a flasque presheaf is not in general flasque. But one may moreover note that $\Sing^k$ is an \emph{epipresheaf}, i.e.\ a presheaf which satisfies gluing, but not unique gluing. Now if $F$ is any epipresheaf and $X$ is paracompact Hausdorff, then one can show that the natural map $F(X) \to \underline F(X)$ is surjective \cite[Ch.\ I, Theorem 6.2]{bredon}. It follows in particular the sheafification of a flasque epipresheaf on a \emph{hereditarily paracompact} Hausdorff space is flasque. If $X$ is not hereditarily paracompact Hausdorff it may not be true that $\underline \Sing$ is flasque, and a slightly more elaborate argument is needed: namely, it will always be true that $\Sing^0 = \underline \Sing^0$ is a flasque sheaf of rings, and $\underline \Sing^k$ is a module over $\underline \Sing^0$. On a paracompact Hausdorff space one knows that flasque sheaves are soft, modules over soft sheaves of rings are fine, and fine sheaves are acyclic (see \cite[Ch. II, \S 9]{bredon}; hence again $\underline \Sing$ is an acyclic resolution.  For (iv), one may identify the image of $\Gamma(X,\Sing) \to \Gamma(X,\underline \Sing)$ with the filtered colimit over all open coverings $\mathfrak U$ of $X$ of the complexes of `$\mathfrak U$-small cochains'. Since $X$ is paracompact Hausdorff, $\Gamma(X,\Sing) \to \Gamma(X,\underline \Sing)$ is onto, and then the result follows from the `theorem of small simplices' \cite[Ch.\ 4, Sect.\ 4, Theorem 14]{spanier}.
\end{para}

\begin{para}As mentioned in \S\ref{intro paragraph}, Sella proved a comparison isomorphism without assuming $X$ paracompact Hausdorff. His approach is to turn $\Sing$ into a sheaf not by means of sheafification, but instead by a very carefully constructed procedure of taking a colimit over smaller and smaller singular simplices. It turns out that the result is a resolution of the constant sheaf by products of skyscraper sheaves, whose global sections can be compared to the singular cochains on $X$. 

\end{para}\begin{para}
A homotopy theorist, presented with the proof outline of \S\ref{outline}, might proclaim that step (ii) is `morally wrong', and therefore also (iii) and (iv). Indeed one could object that \textbf{we should not have to sheafify $\Sing$, since it is already a sheaf}. As this statement is obviously false as stated --- the presheaves $\Sing^k$ are certainly not sheaves --- I should explain what it means. The point is that we may think of $\Sing$ as a presheaf taking values in the \emph{$\infty$-category} of cochain complexes (a higher categorical enhancement of the classical derived category of abelian groups). There is a well defined notion of what it means for a presheaf taking values in an $\infty$-category to be a sheaf, and $\Sing$ is always a sheaf in this higher categorical sense. In the rest of this note we will explain this statement, and how it implies the comparison isomorphism between sheaf cohomology and singular cohomology: given that $\Sing$ is a sheaf in the sense of higher category theory, we only need  step (i) of the proof outline of \S\ref{outline}, and then the proof is finished. 
\end{para}

\section{Hypersheaves and cohomology}

\begin{defn}\label{sheafdef}
Let $M$ be a model category (or a complete $\infty$-category), and $X$ a topological space. A presheaf $F \colon \open(X)^\op \to M$ is called a \emph{hypersheaf} if for any open hypercover $V_\bullet \to U$  of an open subset of $X$, the map
$$ F(U) \longrightarrow \holim F(V_\bullet) $$
is a weak equivalence.
\end{defn}

\begin{rem}Every open cover of a topological space is in particular a hypercover, and the reader uncomfortable with hypercovers may consider only ordinary open covers throughout without losing much.\footnote{In fact one can also define the weaker notion of an \emph{$\infty$-sheaf}, by imposing the descent condition only for open covers, rather than general hypercovers. Conventionally in higher category theory these would be just called \emph{sheaves}, not $\infty$-sheaves, but this terminology seems potentially confusing here. It is an insight of Lurie \cite{highertopostheory} that $\infty$-sheaves are for several reasons more fundamental than hypersheaves; in particular, the $\infty$-category of hypersheaves can be recovered from the $\infty$-category of $\infty$-sheaves by the process of hypercompletion. For bounded below presheaves of cochain complexes the notions of $\infty$-sheaf and hypersheaf are in fact equivalent, so that in principle we could have worked with $\infty$-sheaves throughout. We have opted to work with hypersheaves since we will check quasi-isomorphisms on stalks, which in general is not valid for $\infty$-sheaves but always for hypersheaves.} To be very concrete, suppose we are given a presheaf $F$ on $X$ valued in cochain complexes. If $\mathfrak U$ is an open cover of an open subset $U$ of $X$, then we may form the usual \v{C}ech complex $\check{C}(\mathfrak{U},F)$. It is a double complex, one differential being the \v{C}ech differential and one being the internal differential coming from $F$ being a presheaf of cochain complexes. The hypersheaf axiom of \cref{sheafdef} says in this case that the natural map $F(U) \to \check{C}(\mathfrak{U},F)$ is a quasi-isomorphism. 
\end{rem}

\begin{para}\label{projective}One can define sheaf cohomology in terms of hypersheaves.  It is natural to formulate this model-categorically, in terms of the \emph{local projective model structure on presheaves of cochain complexes}, described e.g.\ in \cite{jardine,hinichsheaves,choudhurygallauer}. It is the differential graded analogue of the more familiar local model structure on \emph{simplicial} presheaves. The homotopy category of this model category is the classical unbounded derived category of sheaves of $R$-modules.\end{para}

\begin{prop}\label{model structure}Let $X$ be a topological space, $R$ a commutative ring. There is a model structure on the category of presheaves of unbounded cochain complexes of $R$-modules on $X$, such that the weak equivalences are the maps inducing quasi-isomorphisms on all stalks, and the fibrations are the pointwise surjections whose kernel is a hypersheaf. 
\end{prop}

\begin{para}\label{barwick para}One way of understanding the model structure of \cref{model structure} is as follows. 
The category $\mathrm{Ch}_R$ has a standard model structure whose fibrations are degreewise surjections, and weak equivalences are quasi-isomorphisms. Then we also have the projective model structure on the category of presheaves on $X$ with values in $\mathrm{Ch}_R$ (which in this context is often called the global projective model structure), with fibrations and quasi-isomorphisms defined pointwise. The local projective model structure may then be seen as the left Bousfield localization of the global projective model structure at the class of all hypercovers. Existence of the left Bousfield localization can be seen from the general theory of combinatorial model categories, due to Smith (unpublished), for which canonical references include Barwick \cite{barwick} and Lurie \cite[Appendix A.2]{highertopostheory}. More specifically, these are left proper combinatorial model categories, and therefore admit left Bousfield localizations at arbitrary sets of maps; in particular we may localize at any dense set of hypercovers. See in particular Barwick \cite[Section 4, Application IV]{barwick}, where the analogous procedure is carried out for \v{C}ech covers instead of hypercovers, giving a model category of $\infty$-sheaves rather than hypersheaves. The fact that the equivalences in this Bousfield localization are precisely the stalkwise equivalences is not immediate, and is due to Dugger--Hollander--Isaksen \cite{duggerhollanderisaksen} in the setting of simplicial presheaves. See \cite{choudhurygallauer} for the analogous result for cochain complexes. \end{para}

\begin{para}\label{def-cohomology}We can now give a definition of sheaf cohomology. Consider the constant presheaf/global sections adjunction
$$ \mathrm{const} : \mathrm{Ch}_R \leftrightarrows \mathrm{PSh}(X,\mathrm{Ch}_R) : \Gamma(X,-) $$
between the categories of cochain complexes, and presheaves of cochain complexes on $X$. In the previous paragraph we explained that when $\mathrm{Ch}_R$ is given its standard model structure and $\mathrm{PSh}(X,\mathrm{Ch}_R)$ the local projective model structure, then this is a Quillen adjunction (in fact a composite of two Quillen adjunctions, one with respect to the global projective model structure, and one Bousfield localization). Then the general theory of model categories tells us how to define, and how to compute, a right derived functor $R\Gamma(X,F)$: first choose a weak equivalence $F \to F'$, where $F'$ is fibrant (i.e.\ a hypersheaf), and define $R\Gamma(X,F)=\Gamma(X,F')$. Up to homotopy the result is independent of the choice of fibrant replacement $F'$, according to Ken Brown's lemma. The cohomology of the complex $R\Gamma(X,F)$ is then by definition the cohomology $H^\bullet(X,F)$. 
\end{para}

\begin{para}Let us for the reader's convenience make explicit the comparison with more standard definitions of sheaf cohomology. Conventionally one would define the sheaf cohomology $H^\bullet(X,F)$ by choosing a quasi-isomorphism $F \to F'$ to a complex of flabby sheaves, and taking global sections of $F'$. To deduce that our definition of sheaf cohomology coincides with this one, we need to know that complexes of flabby sheaves are also hypersheaves.  \end{para}

\begin{prop}\label{hypersheaf0}
Let $R$ be a commutative ring, and let $K$ be a bounded below cochain complex of flabby sheaves of $R$-modules on $X$. Then $K$ is a hypersheaf. 
\end{prop}

\begin{proof}Since flabbiness is preserved when restricting to an open subset, it is enough to verify the hypersheaf axiom for an open hypercover $\epsilon\colon V_\bullet \to X$. For every such hypercover there is a complex of sheaves $\check{C}_\epsilon(K)$ on $X$ such that $\holim K(V_\bullet) = \Gamma(X,\check{C}_\epsilon(K))$, with a map $K \to \check C_\epsilon(K)$ inducing the map $K(X) \to \holim K(V_\bullet)$. Explicitly, $\check C_\epsilon(K)$ is the homotopy limit of the cosimplicial object which in degree $n$ is given by $K$ restricted to $V_n$, pushed forward to $X$. We note in passing that the homotopy limit may be computed as the totalization since every cosimplicial chain complex is Reedy fibrant. We claim now that:
\begin{enumerate}
\item $K \to \check C_\epsilon(K)$ is a quasi-isomorphism,
\item $\check C_\epsilon(K)$ is a bounded below complex of flabby sheaves,
\end{enumerate}
from which the result follows, since a quasi-isomorphism between bounded below complexes of flabby sheaves induces a quasi-isomorphism between global sections (this is what makes traditional sheaf cohomology well defined!). For the first point it is enough to check on stalks, which reduces the claim to the case that $X$ is a point. But a hypercover of a point is a contractible Kan complex and in particular its cohomology is trivial. The second point is straightforward, using that flabbiness of a sheaf is preserved by taking products, restricting to an open subset, and pushing forward. \end{proof}

\begin{rem}An identical proof shows that unbounded complexes which are $K$-flabby in the sense of Spaltenstein \cite{spaltenstein} are hypersheaves. Hence for unbounded complexes of sheaves the notion of cohomology defined here coincides with Spaltenstein's. \end{rem}

\section{The proof}

\begin{prop}\label{hypersheaf}The presheaf $\Sing$ is a hypersheaf, for an arbitrary topological space $X$. 
\end{prop}

\begin{proof}
Let $V_\bullet \to U$ be an open hypercover. We need to prove that the induced map
$$ C^\bullet_\sing(U,A) \longrightarrow \holim C^\bullet_\sing(V_\bullet,A) $$
is a quasi-isomorphism. According to a result of Dugger--Isaksen \cite[Theorem 1.2]{duggerisaksen}, the map $\hocolim V_\bullet \to U$ is a weak homotopy equivalence, and so we need to show that 
$$ C^\bullet_\sing(\hocolim V_\bullet,A) \longrightarrow \holim C^\bullet_\sing(V_\bullet,A) $$
is a quasi-isomorphism.

In classical category theory it is well known that representable functors take colimits to limits. The analogous statement in model category theory is that the `function complex' $\mathbb R \mathrm{map}(-,X)$ takes homotopy colimits to homotopy limits. And the functor of singular cochains can be described as such a function complex: if we denote by $HA$ the Eilenberg-MacLane spectrum of the ring $A$, then under the equivalence between chain complexes over $A$ and $HA$-module spectra, $C^{-\bullet}_\sing(U,A)$ is equivalent to the function spectrum $F(U_+,HA)$. See \cite[Chapter IV]{ekmm} and \cite{shipley}.  The result follows. \end{proof}

\begin{rem}The last equivalence in the preceding proof is a cochain-level version of the perhaps more familiar statement that cohomology is representable: $H^n_\sing(U,A)  \cong [U_+,K(A,n)]$, where $[-,-]$ denotes weak homotopy classes of based maps.\end{rem}


\begin{thm}\label{mainthm}Let $X$ be a topological space, cohomologically locally connected with respect to $A$. There is a canonical isomorphism $H^\bullet_\sheaf(X,A) \cong H^\bullet_\sing(X,A)$. 
\end{thm}

\begin{proof}Consider the constant presheaf $A$ on $X$, and the map 
$ A \to \Sing.$
Recall that $ A \to \Sing$ is a weak equivalence precisely if it induces a quasi-isomorphism on stalks (\cref{model structure}). Since the stalk of $\Sing$ at $x$ is the cochain complex $\varinjlim_{U \ni x}C^\bullet_\sing(U,A)$, we see that $A \to \Sing$ is a weak equivalence if and only if $X$ is cohomologically locally connected. 
Moreover, by \cref{hypersheaf} we know that $\Sing$ is a hypersheaf, hence a fibrant replacement of $A$ if $X$ is cohomologically locally connected. So in this case we can compute the sheaf cohomology of $X$ as the global sections of $\Sing$, which is nothing but the singular cochain complex of $X$ with coefficients in $A$. \end{proof}

\begin{rem}A reader may wonder where the point-set topology happened in this proof. Surely one can not prove a theorem in general topology merely by mumbling about hypersheaves. The answer is that the result of Dugger--Isaksen used in \cref{hypersheaf} is a genuinely nontrivial point-set theorem: indeed, their result says that if $V_\bullet \to U$ is \emph{any} open hypercover of an \emph{arbitrary} topological space $U$, then $\hocolim V_\bullet$ is weakly homotopy equivalent to $U$. The conclusion of their result says that the homotopy colimit of a certain diagram coincides with its usual colimit, which we usually only expect for suitably \emph{cofibrant} diagrams, which is far from the case here.\end{rem}

\begin{rem}The theory of sheaf cohomology has traditionally been approached using the tools of homological algebra and derived categories. The homotopy-theoretic approach to sheaf theory used here was developed to treat situations where the methods of homological algebra break down, such as sheaves of simplicial sets or spectra. A full history of the subject would be out of place here but contributions of \cite{sga4,brown,browngersten,thomason,jardine0,duggerhollanderisaksen,highertopostheory} 
should be mentioned. This note illustrates that these methods can be useful and clarifying also in very classical and linear situations. 
\end{rem} 

\section{$E_\infty$-structure on cochains}

\begin{para}Suppose now that $A$ is not just an abelian group but a commutative ring. The argument can be jazzed up to prove that $\mathrm R\Gamma(X,A)$ and $C^\bullet_\sing(X,A)$ are quasi-isomorphic as \emph{$E_\infty$-algebras} in $\mathrm{Ch}_A$. In fact they are canonically isomorphic in the associated $\infty$-category. One way to prove this is to observe that the weak equivalence $A \to \Sing$ considered in the proof of \cref{mainthm} is in fact a morphism of \emph{presheaves of $E_\infty$-algebras} in cochain complexes. It would then be enough to argue that $\Sing$ is fibrant for a suitable model structure on presheaves of $E_\infty$-algebras on $X$, and that the adjunction of \S\ref{def-cohomology} is a \emph{symmetric monoidal} Quillen adjunction. This is indeed possible, but since not all of the details are in an easily quotable form in the literature, we will give a slightly different argument.\end{para}

\begin{para}Recall that the totalization functor from cosimplicial cochain complexes to cochain complexes is monoid\-al, but not symmetric monoidal. It is however symmetric monoidal in a homotopical sense: it induces a symmetric monoidal functor between the associated $\infty$-categories. A more classical way of expressing this homotopy symmetricity property uses the \emph{Eilenberg--Zilber operad} $\mathsf{Z}$ of Hinich--Schechtman \cite{hinichschechtman}, which has the following properties:
\begin{enumerate}[(i)]
\item If $\mathsf O$ is an operad in cochain complexes, then the totalization of a cosimplicial $\mathsf O$-algebra is an $\mathsf O \otimes \mathsf{Z}$-algebra.
\item There is a canonical quasi-isomorphism of operads $\mathsf{Z} \to \mathsf{Com}$. 
\end{enumerate}
We remark first of all that these facts explain why $\Sing$ is a presheaves of $E_\infty$-algebras. Indeed, the singular cochains of any space are manifestly obtained as the totalization of a cosimplical commutative ring, hence form a $\mathsf Z$-algebra by (i), and by (ii), any $\mathsf Z$-algebra may be considered as an $E_\infty$-algebra. 

We will compare $\Sing$ with the \emph{Godement resolution}, and its natural $E_\infty$-algebra structure. The Godement resolution comes with a natural weak equivalence $F \stackrel \sim \longrightarrow \mathrm{Gode}(F)$. The functor $F \mapsto \mathrm{Gode}(F)$ is the composition of a symmetric monoidal functor from presheaves on $X$ to cosimplicial presheaves on $X$, and the functor of totalization, which in particular means that the Godement construction applied to a presheaf of $\mathsf O$-algebras is a presheaf of $\mathsf O \otimes \mathsf{Z}$-algebras.

But this means now that we may start from the local weak equivalence $A \to \Sing$ of presheaves of $\mathsf Z$-algebras, and apply the Godement functor to obtain a commuting square 
	\[	\begin{tikzcd}
		A \arrow[r,"\sim"]\arrow[d,"\sim"]& \mathrm{Gode}(A) \arrow[d,"\sim"] \\
		\Sing \arrow[r,"\sim"]& \mathrm{Gode}(\Sing)	\end{tikzcd} \]
of local weak equivalences of presheaves of $\mathsf Z \otimes \mathsf Z$-algebras. All entries in this diagram except the top left corner are hypersheaves. It follows that taking global sections gives a zig-zag of quasi-isomorphisms of $\mathsf Z \otimes \mathsf Z$-algebras, hence of $E_\infty$-algebras,
$$ C^\bullet_{\mathrm{sing}}(X,A) = \Gamma(X,\Sing) \stackrel\sim\longrightarrow \Gamma(X,\mathrm{Gode}(\Sing))\stackrel\sim\longleftarrow \Gamma(X,\mathrm{Gode}(A)) = \mathrm R\Gamma(X, A), $$
which gives the result. \end{para}

\printbibliography

\end{document}